\documentclass[11pt,a4paper, reqno]{amsart}
\usepackage{fullpage,amssymb}
\usepackage{hyperref}
\newtheorem{theorem}{Theorem}[section]
\newtheorem{lemma}[theorem]{Lemma}
\theoremstyle{remark}

\usepackage[initials]{amsrefs}
\makeatletter
\@namedef{subjclassname@2010}{%
\textup{2010} Mathematics Subject Classification}
\makeatother

\theoremstyle{definition}
\newtheorem{conjecture}[theorem]{Conjecture}

\newtheorem{problem}[theorem]{Problem}

\numberwithin{equation}{section}

\newcommand{\F}{\mathbb{F}_q}
\newcommand{\f}{\mathbb{F}_4}

\newcommand{\la}{\lambda}
\newcommand{\m}{\mathrm{mod}\;}
\newcommand{\g}{\gamma}
\newcommand{\A}{\alpha}

\usepackage{xcolor}
\usepackage{listings}
\lstdefinelanguage{pari-gp}
{
  morekeywords={my,for,return,if,fordiv,forprime,while},
  sensitive=true,
  morecomment=[l]{\\\\},
  morecomment=[s]{/*}{*/},
  morestring=[b]"
}
\begin{document}
\baselineskip=17pt


\title[Representing elements in finite fields ...]{Finite fields whose members are the sum of \\ a potent and a 4-potent}

\author[S.D.~Cohen]{Stephen D. Cohen}
\address{6 Bracken Road, Portlethen, Aberdeen AB12 4TA, Scotland, UK}
\email{stephen.cohen@glasgow.ac.uk}
\author[P.V.~Danchev]{Peter V. Danchev}
\address{Bulgarian Academy of Sciences, Institute of Mathematics and Informatics, Sofia 1113, Bulgaria}
\email{danchev@math.bas.bg}
\author[T.O.e~Silva]{Tom\'as Oliveira e Silva}
\address{IEETA/LASI, Institute of Electronics and Informatics Engineering of Aveiro}
\address{DETI, Department of Electronics, Telecommunications and Informatics, University Of Aveiro, 3810-193 Aveiro, Portugal}
\email{tos@ua.pt}

\date{\today}

\begin{abstract} We classify those finite fields $\F$ whose members are the sum of an $n$-potent element with $n>1$  and a 4-potent element. It is shown that there are precisely ten non-trivial pairs $(q,n)$ for which this is the case.   
This continues a recent publication by Cohen-Danchev et al. in Turk. J. Math. (2024) in which the tripotent version was examined in-depth as well as it extends recent results of this branch established by Abyzov-Tapkin in Sib. Math. J. (2024).
\end{abstract}



\maketitle

\section{Introduction and Motivation}

Throughout this article, let $\F$ denote the finite field of order a prime power $q=p^v$, where $p$ is a prime and $v$ is a positive integer.  Further,  let  $n$ be an integer strictly exceeding 1. If $a \in \F$ is such that $a^n=a$, then $a$ is called an {\em  $n$-potent}, and we denote the set of all $n$-potents in $\F$ by $C_n$. Obviously, if $(q-1)|n$, then $C_n=\F$. Moreover, if $m$ denotes $\mathrm{gcd}(n,q-1)$, then $C_m=C_n$. Hence, the full range of non-trivial sets $C_n$ is obtained by restricting $n$ to the integers for which  $(n-1)|(q-1)$ and $n>1$.
\medskip

The original motivation for this line of investigation comes from problems involving the notion of potents in ring theory and, in particular, in some aspects of matrix theory. This has inspired the following general problem.

\medskip

\noindent{\bf General Problem.} Find those finite fields each of whose members is the sum of an $m$-potent and an $n$-potent, where $m > 1$ and $n >1$ are arbitrary integers.

\medskip

A discussion of the case in which $m=2$ relating to the sum of an $n$-potent and a 2-potent (= idempotent) was given in \cite{at-2021} where a little finite field theory was used. Likewise, the article \cite{acdt-2023} explored the  case with $m=3$ relating to a sum between an $n$-potent and a 3-potent (= tripotent). This involved more serious finite field theory which is the subject of \cite{cohen-2026}. (See also  \cite{at1}, \cite{at2} for some other closely related material.) It follows that is sensible to isolate the following problem as a worthy object of study in finite field theory. Concretely, we are interested in whether $C_n+C_4=\F$ whenever $n\geq 2$, formulated more exactly as follows:  

\medskip

\noindent{\bf Restricted Problem.} Characterize those finite fields for which every element is a sum of an $n$-potent (where $n>1$) and a 4-potent. As noted above, we can assume $(n-1)|(q-1)$, and also $n<q$, to avoid the trivial case in which $n=q$ (in which case $C_n=\F$).

\medskip

By way of comparison, the existence theorem implied by Theorem 1.2 and Lemma 2.1 of \cite{acdt-2023} can be stated as follows.

\medskip
\begin{theorem}\label{tripotent}
Suppose $q$ is a prime power and $n$ is a positive integer such that $n<q$ and $n-1$ divides $q-1$. Then, every element of $\F$ is the sum of a tripotent and an $n$-potent if, and only if, $n=\frac{q-1}{2}$ and $q \in \{3,5,7,9\}$.
\end{theorem}

\medskip

The main result of this paper, motivating our writing of the present paper, is to give a complete answer to the stated above Restricted Problem. Specifically, we will prove in the sequel the following statement.

\begin{theorem}\label{main}
Suppose $q$ is a prime power and $n$ is a positive integer such that $n<q$ and $n-1$ divides $q-1$. Then, every element of $\F$ is the sum of a $4$-potent and an $n$-potent if, and only if, $(q,n)$ is one of the pairs in the set
\[\{4,2), (7,3),(7,4), (13,7), (19,10), (25,13), (31,16), (43,22), (49,25), (103,52)\}.\]
\end{theorem}

\medskip

The main part of the proof of Theorem~\ref{main} is the theoretical discussion in Section \ref{principal}, but this is supplemented by vital computations described in the Appendix. Some closely related applications of Theorem~\ref{main} in ring theory may be given in further work, and perhaps there will be relevant applications of these results to other areas of algebra such as the matrix theory.

\section{Principal Results}\label{principal}

Given a prime power $q=p^v$, the question we explore is to discern values of $n<q$ with $(n-1)|(q-1)$ for which every element $\g$ of $\F$ can be written as $\A$ + $\beta$, where $\A \in C_n$ and $\beta \in  C_4$, i.e., whether $\displaystyle{\bigcup_{\beta \in C_4}(C_n+\beta) = \F}$.

In full generality $q$ may be odd or even: in the latter case $q=2^v$. In much of the discussion, however, $q$ will be odd, in which case let $\la$ denote the quadratic character on $\F$. Thus, $\la(0)=0$ and otherwise $\la(\g)= \pm1$ according as $\g$ is a non-zero square or non-square.

\medskip

First, suppose $q \not \equiv1 (\m 3)$. Then, $\F$ contains no cube root of unity besides $1$ and $C_4=C_2= \{0,1\}$.   Further, suppose

\begin{equation} \label{2P}
C_n\cup(C_n+1)=\F, \quad n<q, \; (n-1)|(q-1).
\end{equation}
Then, easily $q$ is odd and $n-1=(q-1)/2$, i.e., $n=(q+3)/2$, and $C_n$ comprises the set of all squares in $\F$.  Thus, for every member $\g \in \F$, either $\g$ or $\g+1$ must be a square. In other words, there cannot be a consecutive pair of non-squares in $\F$.  On the other hand, the number of pairs of consecutive elements of $\F$ that are both non-squares can be expressed as
\[ \frac{1}{4} \sum_{{\A  \neq 0, -1}}(1-\la(\A))((1- \la(\A+1)), \]
where the sum is over all $\A \in \F$  except $0$ or $-1$.
As in the method used in \cite{acdt-2023}, this number is exactly $\frac{1}{4}(q-2+\lambda(-1)) \ge (q-3)/4$, and so is positive provided $q>3$. We conclude that $(\ref{2P})$ is impossible in this case.

\medskip

From now on suppose $3|(q-1)$. In particular, if $q=2^v$, then $v$ is even and $C_4$ is the subfield $\f$. If $q$ is odd, then $\la((-3)=1$ and in every case $\F$ contains 3 cube roots of unity. Thus, $C_4= \{0, 1, z, z^2=-z-1\}$, where $z$ is a primitive cube root of unity in $\F$. Since $z=1/z^2$, then $\la(z)=1$ when $q$ is odd. Let $y=z^2=-1-z$. Now, the assumption $({}\ref{4P})$ takes the form
\begin{equation} \label{4P}
C_n \cup (C_n+1) \cup(C_n+z)\cup(C_n+y)= \F, \quad n<q, (n-1)|(q-1).
\end{equation}
Therefore, evidently, $n \geq (q-1)/4$. It follows that the possible value of $n$ that remain are $\frac{q-1}{i}+1$, where $i=2, 3,4$. Here, the values of $i=2, 4$ can only occur if $q$ is odd.

\medskip

\subsection{$q$ odd and $n=\frac{q+3 }{4}$\label{q-1/4}}

In considering $(\ref{4P})$, we first dispose of the case in which $i=4$; thus, $n-1=(q-1)/4$, where necessarily $4|(q-1)$, i.e., $q \equiv 1 (\m 12)$ and $C_n$ is the set of all 4th powers in $\F$. Since $4|q-1$, then $C_4 \subseteq C_n$. Since $0 \in C_n$, then $1 \in C_n\cap(C_n+1)$. Further, $z\in C_n\cap (C_n+z)$ and $y \in C_n\cap( C_n+y)$. Finally, $-1 =z+y=y+z\in ( C_n+z)\cap( C_n+y)$. Now, the cardinality of each set $C_n+k, (k \in C_4$) is $(q+3)/4$, so that the cardinality of $C_n \cup (C_n+1) \cup(C_n+z)\cup(C_n+y)$ is at most $4(\frac{q+3}{4}) - 4=q-1<q$. Hence, (\ref{4P}) cannot hold in this case.

\medskip

\subsection{$q$ odd and $n=\frac{q+ 1}{2}$}\label{q-1/2}

Next, suppose $n-1= (q-1)/2$, i.e., $n=(q+1)/2$ and $C_n$ is the set of squares (including $0$) in $\F$. Here $(\ref{4P})$ would imply that, given any member $\g \in \F$, at least one of $\g,\g+1, \g+z,\g+y $ is a square. In other words, not all of these four elements are non-squares. To this end, define $N_q$ to be the number of elements $\A \in \F$ such that each of $\A -1, \A-z, \A-y$ is a non-square. Then,
\[N_q=\frac{1}{16}\sum_{\alpha \notin C_4}(1- \lambda(\A))(1-\lambda(\A-1))(1- \lambda(\A-z))(1-\lambda(\A-y)).\]

Now, let
\[S_1= \sum_{\alpha \notin C_4}\lambda(\alpha),\; S_2= \sum_{\alpha \notin C_4}\lambda(\alpha-1),\;S_3= \sum_{\alpha \notin C_4}\lambda(\alpha-z),\;S_4= \sum_{\alpha \notin C_4}\lambda(\alpha-y);\]
\[ T_1= \sum_{\alpha \notin C_4}\lambda(\alpha(\alpha-1)) ,\; T_2= \sum_{\alpha \notin C_4}\lambda(\alpha(\alpha-z)), \; T_3= \sum_{\alpha \notin C_4}\lambda(\alpha(\alpha-y)) ,  \]
\[ T_4= \sum_{\alpha \notin C_4}\lambda((\alpha-1)(\alpha-z)) ,\; T_5= \sum_{\alpha \notin C_4}\lambda((\alpha-1)(\alpha-y)), \; T_6= \sum_{\alpha \notin C_4}\lambda((\alpha-z)(\alpha-y))  ; \]

\[ U_1= \sum_{\alpha \notin C_4} \la(\alpha(\alpha-1)(\alpha-z)), \quad U_2= \sum_{\alpha \notin C_4} \la(\alpha(\alpha-1)(\alpha-y)),\]
\[ U_3= \sum_{\alpha \notin C_4} \la(\alpha(\alpha-z)(\alpha-y)), \quad U_4= \sum_{\alpha- \notin C_4} \la((\alpha-1)(\alpha-z)(\alpha-y)),\]

\[ \quad V= \sum_{\alpha \notin  C_4} \lambda(\alpha(\alpha^3-1)) .\]

Thus,
\begin{equation} \label{Neq}
N_q=\frac{1}{16}\left(q-4- \sum_{i=1}^4S_i+\sum_{i=1}^6T_i- \sum_{i=1}^4U_i + V\right) .
\end{equation}

\medskip

We proceed to evaluate or estimate these various sums, beginning with the $S_i$. We have
\[ S_1=\sum_{\A \neq 1, y,z}\la((\A)=-\lambda(1)- \lambda(z)- \lambda(y) =-3,\]
\[S_2=\sum_{\A\ne0,y,z}\lambda(\alpha-1)
=\lambda(-1) - \lambda(z-1)-\lambda(y-1) ) =- \lambda(-1)-\lambda(z-1)- \lambda(-z-2).\]
Similarly,
\[S_3=-\lambda(-z)- \lambda(1-z)- \lambda(-2z-1), \]
\[S_4=-\lambda(1+z)- \lambda(z+2)-\lambda(1+2z).\]

\medskip

To find expressions for the $T_i$, we recall that, if $f$ is a monic quadratic polynomial in $\F[x]$ (with distinct roots), then thanks to \cite[Theorem 2.1.2]{BEW98} we receive that
\[ \sum_{\g  \in \F}\la(f(\g))=-1.\]
Accordingly,
\begin{equation*}
\begin{split}
T_1= \sum_{\alpha \neq y,z}\lambda(\alpha(\alpha-1))=-1-\lambda(z(z-1))-\lambda((-z-1)(-z-2)
&= -1-\lambda(-2z-1) -\lambda(2z+1);
\end{split}
\end{equation*}

\begin{equation*}
\begin{split}
T_2&= \sum_{\alpha \ne 1,y}\lambda(\alpha(\alpha-z)) =-1-\lambda(1-z))-\lambda((-z-1)(-z-2)) =-\lambda(1-z)- \lambda((z+1)(2z+1))
\\ &=-1-\lambda(1-z)-\lambda(z-1),
\end{split}
\end{equation*}
since $z^2=-z-1$.

Similarly,
\begin{equation*}
T_3 = \sum_{\alpha \neq 1, z}\lambda(\alpha(\alpha+z+1))=-1-\lambda(z+2)-\lambda(-z-2);
\end{equation*}

\begin{equation*}
T_4 = \sum_{\alpha \ne 0,y}\lambda((\alpha-1)(\alpha-z))=-1-2\lambda(z);
\end{equation*}

\begin{equation*}
\begin{split}
T_5 = \sum_{\alpha \ne 0,z}\lambda((\alpha-1)(\alpha+z+1))&=-1-\lambda(-z-1)-\lambda(-3z-3)\\
&=-2-\-\la(z+1),
\end{split}
\end{equation*}
since $\la(-3)=1$ and $\la(-z-1)=\la(z^2)=1$;
\begin{equation*}
T_6= \sum_{\alpha \ne 0,1}\lambda((\alpha-z)(\alpha+z+1))=-1-\lambda(1)-\lambda(3)=-2-\lambda(-1).
\end{equation*}

\medskip
Aggregating the  working so far, after much cancelation, we conclude simply that

\begin{equation}\label{STterms}
N_q=\frac{1}{16}\left(q-10 +\la(-1) - \sum_{i=1}^4U_i + V\right).
\end{equation}

Beyond $(\ref{STterms})$, we have to estimate the $U$ and $V$ terms using the classic Weil's theorem (see, e.g., \cite{sch76} for an elementary proof).

\begin{lemma}\label{weil}
Suppose $\chi$ is a multiplicative character of order $d > 1$ in $\F$. Further, suppose $f(x) \in \F[x]$ is a polynomial whose set of zeros in its splitting field over $\F$ has cardinality $m$, and that $f$ is not a constant multiple of a $d$-th power. Then, 
\[\sum_{\g \in\F} \chi(f(\g)) \leq (m-1) \sqrt{q}.\]
\end{lemma}

It now follows from Lemma {\ref{weil} that $U_1= \sum_{\alpha \neq y}\la(\alpha(\alpha-1)(\alpha-z))$ has absolute value at most $2 \sqrt{q}+1$, so that $-U_1 \ge -2 \sqrt{q} -1$. Similarly, $-U_i \geq -2 \sqrt{q}-1$, for  $i \leq 4$. Likewise, $V= \sum_{\alpha \in \F}\la(\alpha(\alpha^3-1)) \ge -3 \sqrt{q}$.

\medskip

In fact, we can improve the estimate for $\displaystyle {U_4=\sum_{\A \ne 0}\la(\A^3-1)}$. Let $\eta$ be a cubic character on $\F$. Setting $\A^3=\beta$, we see that
\[U_4= \frac{1}{3}\sum_{i=0}^2\sum_{\beta\neq0}\big(\la(\beta-1)\eta(\beta)\Big).\]
Hence,
\[U_4=\frac{\la(-1)}{3}(-1+J(\la, \eta)  +J(\la,\eta^2)),\]
where $J(\chi,\la)$ is the Jacobi sum $\displaystyle{\sum_{x \in \F}\chi(x)\eta(1-x)}$ (see \cite{BEW98}). Since the Jacobi sums have absolute value $\sqrt{q}$, we deduce that $|U_4|\le \frac{1+2\sqrt{q}}{3}$.
It next follows from $(\ref{STterms})$ and the above bounds that
\[ N_q \geq  \frac{1}{16}\Big( q- \frac{29}{3} \sqrt{q} -16\Big),\]
because $3\times 2 +\frac{2}{3}+3=\frac{29}{3} $ and where the constant term $16$ could be reduced slightly. Certainly, $N_q$ is guaranteed to be positive whenever $q > 124$.

\medskip

We summarize the outcome of Section \ref{q-1/2} in the following assertion.

\begin{theorem}\label{i=2}
Suppose $q$ is an odd prime power and $n=\frac{q-1}{2}$. Suppose $q>124$. Then, not every element of $\F$ can be the sum of a $4$-potent and an $n$-potent.
\end{theorem}

\subsection{$q$ odd or even and $n=\frac{q+ 2}{3}$}\label{q-1/3}

Finally, suppose that $n-1= (q-1)/3$, i.e., $n= (q+2)/3$ and $C_n$ is the set of cubes (including $0$) in $\F$. In particular, in the case when $q$ is even (i.e., $q=2^v$ with $v$ even), then $C_4$ is simply the subfield $\f$. This time $(\ref{4P})$ would imply that, given any element $\g \in \F$, at least one of $\g,\g+1, \g+z,\g+y $ is a cube.     In other words, not all of these four elements are non-cubes. Let $M_q$ denote the number of elements $\gamma \in \F$ such that each of $\g, \g -1, \g-z, \g-y$ is a  non-cube.  We give an expression for $M_q$ in terms of the characteristic function $\nu$ for non-cubes in $\F^*$ defined in the following lemma.

\medskip
\begin{lemma} \label{noncube}
Suppose $3|(q-1)$ and $\eta$ is a cubic character on $\F$. For non-zero elements $\g \in\F$, define the function $\nu$ by
\[\nu(\g)= \frac{1}{3}(2 - \eta(\g)- \eta^2(\g)). \]
Then,
\[\nu(\g) =\begin{cases}
1&\text{if } \g \text{ is a non-cube in } \F,\\
0&\text{if } \g \text{ is a non-zero cube in } \F.
\end{cases}\]
\end{lemma}

\begin{proof}
Suppose $\g$ is a non-zero cube in $\F$. So, $\eta(\g)=1$ and evidently $\nu(\g) =0$. Thus, $\eta(\g)=\zeta$, a primitive cube root of 1 and the result follows since
$\zeta+\zeta^2=-1$.
\end{proof}

Consulting with Lemma~\ref{noncube}, it follows that
\[M_q=\frac{1}{81}\sum_{\alpha \notin C_4}(2-\eta(\A)-\eta^2(\A))(2- \eta(\A-1)-\eta^2(\A-1))(2- \eta(\A-z)-\eta^2(\A-z))(2-\eta(\A-y)-\eta^2(\A-y)).\]
This yields
\begin{equation} \label{Meq}
M_q=\frac{1}{81}\left(16(q-4)-8 \sum_{i=1}^8S_i+4\sum_{i=1}^{24}T_i-2\sum_{i=1}^{32}U_i + \sum_{i=1}^{16}V_i\right),
\end{equation}
where the $S_i, T_i, U_i, V_i$ are character sums as in (\ref{Neq}). In particular, each of these sums has the form
$\displaystyle{\sum_{\A \notin C_4}\eta(g(\A))} $, where $g$ is a polynomial of degree $j=1,2,3,4$ corresponding to $S_i, T_i, U_i, V_i$, respectively.

As in Section \ref{q-1/2}, each $S_i$ can be evaluated explicitly, but in any case $|S_i \leq 3$.

Next, the first four sums of the $T_i$ have the shape

\begin{equation}\label{Teq}
\sum_{\alpha \ne y,z} \eta^j(\alpha)\eta^k(\alpha-1), \; j,k \in \{1,2\}.
\end{equation}
Here, the sums displayed in $(\ref{Teq})$ correspond to selecting the pair $(0,1)$ from $C_4$. Altogether, there are six choices of pairs from $C_4$ and four associated sums $T_i$ with each pair. Taking into account the two excluded values of $C_4$ in such a sum, we see with the aid of Lemma \ref {weil} that $|T_i|\leq \sqrt{q}+2$. But, in half the cases (for example, in $(\ref{Teq})$ with $j\ne k$), we have $|T_i|\leq 3$, since, for instance,

\begin{equation*}
\begin{split}
\sum_{\alpha \ne y,z} \eta(\alpha)\eta^2(\alpha-1)&= \sum_{\alpha \ne 1} \eta\left(\frac{\alpha}{\alpha-1}\right)  -\eta(3(z+1))-\eta(3z)\\
& =\sum_{\beta \neq 1} \eta(\beta) -\eta(3(z+1))-\eta(3z) = -\eta(3(z+1))- \eta(3z) - 1.
\end{split}
\end{equation*}

The $U_i$ terms can be bounded again utilizing Lemma~\ref{weil}. For example,
\[ \left| \left[\sum_{\alpha \ne y}\eta(\alpha(\alpha-1)(\alpha-z)) \right]+ \eta(3)\right| \leq 2 \sqrt{q}.\]
Finally, each of the $V_i$ terms are bounded absolutely by $3\sqrt{q}$.

\medskip

Putting all this together, we have
\begin{equation} \label{Mbound}
M_q > \frac{1}{81} \Big (16q- 224 \sqrt{q} -560\Big),
\end{equation}
which is positive whenever $ q> 261$.

\medskip

We now summarize the outcome of Section \ref{q-1/3} in the following claim.

\begin{theorem}\label{i=3}
Suppose $q$ is an odd prime power and $n=\frac{q-1}{3}$. Suppose $q>261$. Then, not every element of $\F$ can be the sum of a $4$-potent and an $n$-potent.
\end{theorem}

We finish our work with the final subsection.

\subsection{Completion of proof of Theorem \ref{main}}\label{completion}

Explicit computations (see Appendix) were done for all $q\leq 1000$. This covers the gaps left by Theorems \ref{i=2} and \ref{i=3} and acts as a further check on all values of $q$ up to 1000.

\section{A conjecture and a further problem} \label{conjecture}

There is a natural extension of the Restricted Problem to the case $m=5$. We have not yet undertaken the theoretical aspect of the argument, but have already performed sufficient computations for $q$ up to 10000 to be confident about the veracity of the following conjecture.

\begin{conjecture}\label{FF5}
With the normal conditions, every element of the finite field $\F$ is a sum of a 5-potent and an $n$-potent only for the $(q,n)$ pairs
$(7,4)$, $(9,3)$, $(9,5)$, $(5,2)$, $(5,3)$, $(13,5)$, $(13,7)$, $(17,9)$, $(25,9)$, $(25,13)$, $(29,15)$, $(41,21)$, $(49,25)$, $(53,27)$, $(73,37)$, $(81,41)$, and $(125,63)$.
\end{conjecture}

In a new direction, we also suggest a further problem of some interest and importance.

\begin{problem} Describe those finite fields whose elements are sums of a potent plus a 3-potent plus a 4-potent.
\end{problem}

\appendix
\section{Pari-gp code}

Explicit computations take advantage of the fact that different models of the finite field $F_q$ are isomorphic to each other. This implies that properties that rely only on additions and multiplications proved for one particular model of a given finite field remain valid for all other models of the same finite field. All explicit computations will be done using the pari-gp computer algebra system~\cite{PARI2}, version 2.13.1.

The following pari-gp function receives the order $q$ of a finite field and returns a list with all its elements.
It uses one specific model of the finite field; that same model will be used in all subsequent computations.
\begin{lstlisting}
ff_elements(q)={
  my(f,r);                   /* local variables                 */
  f=vector(q);               /* will hold all field elements    */
  r=ffprimroot(ffgen(q));    /* multiplicative group generator  */
  f[1]=r-r;                  /* store zero                      */
  f[2]=r^0;                  /* store one                       */
  for(k=3,q,f[k]=r*f[k-1];); /* store the other field elements  */
  return(Set(f));            /* return the sorted list          */
};
\end{lstlisting}

The following pari-gp function receives a list of all elements of a finite field and an $n$ value and retuns the list of all $n$-potents.
\begin{lstlisting}
n_potents(f,n)={
  my(l,nl);                  /* local variables                   */
  l=vector(#f);              /* will hold the list of n-potents   */
  nl=0;                      /* number of n_potents               */
  for(k=1,#f,
    if(f[k]==f[k]^n,         /* test if f[k] is an n-potent       */
      nl=nl+1;               /* yes, add it to the list           */
      l[nl]=f[k];
    );
  );
  l=vector(nl,k,l[k]);       /* trim list                         */
  return(Set(l));            /* return the sorted list            */
};
\end{lstlisting}

The following pari-gp function receives the order $q$ of a finite field and an $n$ value and checks if there is any $k$ for which all the elements of the finite field are the sum of a $n$-potent and an $k$-potent. It prints the tuple $(q,n,k)$ when that is so.
As stated in the introduction, it is only necessary to test $k$ values for which $k-1|q-1$.
\begin{lstlisting}
check_one(q,n)={
  my(f,k,pn,pk,s);           /* local variables                   */
  f=ff_elements(q);          /* the finite field elements         */
  pn=n_potents(f,n);         /* the set of n-potents              */
  fordiv(q-1,d,              /* try all divisors of q-1           */
    k=d+1;                   /* the other potent exponent         */
    if(k<q-1,                /* check only proper divisors        */
      pk=n_potents(f,k);     /* the set of k-potents              */
      if(#pn*#pk>=#ff,       /* do we have enouigh potents?       */
        s=setbinop((x,y)->x+y,pn,pk); /* yes, add the two sets    */
        if(#s==q,                     /* all elements?            */
          printf("%d %d %d\n",q,n,k); /* yes, print relevant data */
        );
      );
    );
  );
};
\end{lstlisting}

Finally, the following pari-gp function searches for tuples $(q,n,k)$ for $q$ values up to a given limit, and for a given value of $n$.
\begin{lstlisting}
check_all(n=4,limit=10^3)={
  my(k,dt);                  /* local variables                   */
  dt=getabstime();           /* start measuring execution time    */
  forprime(p=2,limit,        /* try all primes <= limit           */
    k=1;                     /* initil exponent                   */
    while(p^k<=limit,        /* try all prime powers <= limit     */
      q=p^k;                 /* number of finite field elements   */
      check_one(q,n);        /* check this prime power            */
      k++;                   /* next exponent                     */
    );
  );
  dt=getabstime()-dt;        /* measure execution time            */
  printf("done in %.1fs\n",0.001*dt); /* report execution time    */
};
\end{lstlisting}

After these functions are defined, the following code was run.
\begin{lstlisting}
check_all(4,1000);
\end{lstlisting}
It produced the following output which serves associated with the claim in Section \ref{completion}.
\begin{lstlisting}
4 4 2
25 4 13
7 4 3
7 4 4
49 4 25
13 4 7
19 4 10
31 4 16
43 4 22
103 4 52
done in 1.1s
\end{lstlisting}
The statement of Conjecture~\ref{FF5} was done in the same way, by running the code
\begin{lstlisting}
check_all(5,10000);
\end{lstlisting}

\end{document}